\documentclass[12pt, reqno]{amsart}
\usepackage{amsmath,amsopn,amssymb,amsthm}
\usepackage[english]{babel}
\usepackage[backref=page, breaklinks=true,colorlinks=true,linkcolor=blue,citecolor=blue,urlcolor=blue]{hyperref}
\usepackage{graphicx}

\renewenvironment{proof}[1][Proof]{\textbf{#1.} }{\ \rule{0.5em}{0.5em}}

\DeclareMathOperator{\length}{length}
\DeclareMathOperator{\per}{per}
\DeclareMathOperator{\diam}{diam}
\DeclareMathOperator{\bd}{bd}
\DeclareMathOperator{\co}{co}
\DeclareMathOperator{\ext}{ext}
\DeclareMathOperator{\intt}{int}

\renewenvironment{proof}[1][Proof]{\textbf{#1.} }
{\ \rule{0.5em}{0.5em}}
\newtheorem{theorem}{Theorem}
\newtheorem{prop}{Proposition}
\newtheorem{lemma}{Lemma}
\newtheorem{corollary}{Corollary}

\newtheorem{conjecture}{Conjecture}

\theoremstyle{definition}

\newtheorem{remark}{Remark}
\newtheorem{example}{Example}

\newtheorem{problem}{Problem}

\begin{document}

\title
[One property of a planar curve \dots]
{One property of a planar curve whose convex hull covers a given convex figure}
\author{Yu.G.~Nikonorov, Yu.V.~Nikonorova}

\address{Nikonorov\ Yuri\u\i\  Gennadievich\newline
Southern Mathematical Institute of \newline
the Vladikavkaz Scientific Center of \newline
the Russian Academy of Sciences, \newline
Vladikavkaz, Markus st., 22, \newline
362027, RUSSIA}
\email{nikonorov2006@mail.ru}

\address{Nikonorova\ Yulia  Vasil'evna\newline
Volgodonsk Engineering Technical Institute the branch \newline
of National Research Nuclear University ``MEPhI'',\newline
Rostov region, Volgodonsk, Lenin st., 73/94, \newline
347360,  RUSSIA}
\email{nikonorova2009@mail.ru}

\begin{abstract}
In this note, we prove the following conjecture by
A.~Akopyan and V.~Vy\-sotsky: If the convex hull of
a planar curve $\gamma$ covers a planar convex figure $K$, then
$\length (\gamma) \geq  \per (K)  - \diam (K)$.
In addition, all cases of equality in this inequality are studied.

\vspace{2mm}
\noindent
2010 Mathematical Subject Classification:
52A10, 52A40, 53A04.

\vspace{2mm} \noindent Key words and phrases:  planar curve, convex curve, convex hull, diameter, perimeter.
\end{abstract}

\maketitle

\section{Introduction}\label{sect.1}

The authors of \cite{AkVys2017} (along with obtained interesting results) posed the following

\begin{conjecture}[A.~Akopyan and V.~Vysotsky, \cite{AkVys2017}]
Let $\gamma$ be a curve such that its convex hull covers a planar convex figure $K$.
Then $\length (\gamma) \geq  \per (K)  - \diam (K)$.
\end{conjecture}

It should be noted that this hypothesis is confirmed in the case when $\gamma$ is passing through
all extreme points of $K$ (see Theorem 7 in \cite{AkVys2017}).
This note is devoted to the proof of the above conjecture in the general case.
\medskip

We identify the Euclidean plane with $\mathbb{R}^2$ supplied with the standard Euclidean metric~$d$, where $d(x,y)=\sqrt{(x_1-y_1)^2+(x_2-y_2)^2}$.
For any subset $A\subset \mathbb{R}^2$, $\co (A)$ means the convex hull of $A$. For every points $B,C \in \mathbb{R}^2$, $[B,C]$
denotes the line segment between these points.

{\it A convex {\rm(}planar{\rm)} figure} is any compact convex subset of $\mathbb{R}^2$.
We shall denote by $\per(K)$, $\bd (K)$ and $\intt (K)$ respectively the perimeter, the boundary, and the interior of a convex figure $K$.
Note that the perimeter of any line segment (i.e. a degenerate convex figure) is assumed to be equal to its double length.
Note also that the diameter $\diam (K):=\max \left\{ d(x,y)\,|\, x,y \in K \right\}$
of a convex figure $K$ coincides with the maximal distance between two parallel support lines
of $K$.
Recall that an extreme point of $K$ is a point in $K$ which does not lie in any open line segment joining two points of $K$.
The set of extreme points of $K$ will be denoted by $\ext(K)$. It is well-known that $\ext(K)$ is closed and
$K=\co(\ext(K))$ for any convex figure $K\subset \mathbb{R}^2$.

{\it A planar curve} $\gamma$ is the image of a continuous mapping $\varphi:[a,b]\subset \mathbb{R} \mapsto \mathbb{R}^2$.
From now on we will call planar curves simply {\it curves} for brevity, since no other curves are considered in this note.
As usually, the length of $\gamma$ is defined as  $\length (\gamma):=\sup \left\{ \sum_{i=1}^m d(\varphi(t_{i-1}),\varphi(t_i))\right\}$,
where the supremum is taken over all finite increasing sequences $a=i_0<i_1<\cdots <i_{m-1}<i_m=b$ that lie in
the interval $[a,b]$.
A curve $\gamma$ is called {\it rectifiable} if $\length (\gamma) <\infty$.

We call a curve $\gamma\subset \mathbb{R}^2$ {\it convex {\rm(}closed convex{\rm)}} if it is a closed connected subset of the boundary (respectively, the whole boundary)
of the convex hull $\co(\gamma)$ of $\gamma$.

Let us consider the following

\begin{example}\label{ex1}
Suppose that the boundary $\bd (K)$ of a convex figure $K$ is the union of a line segment $[A,B]$ and a convex curve $\gamma$ with the endpoints $A$ and $B$.
Then $K \subset \co(\gamma)$ and $\length (\gamma)=\per (K) - d(A,B)$. Moreover, $\length (\gamma)=\per (K) - \diam(K)$ if and only if $d(A,B)=\diam(K)$.
\end{example}

The main result of this note is the following

\begin{theorem}\label{the1}
For a given convex figure $K$ and for any planar curve $\gamma$ with the property $K\subset \co(\gamma)$, the inequality
\begin{equation}\label{eq.main}
\length (\gamma)\geq \per (K) - \diam(K)
\end{equation}
holds. Moreover, this inequality becomes an equality if and only if
$\gamma$ is a convex curve, $\bd (K)=\gamma \cup [A,B]$, and $\diam(K)=d(A,B)$, where $A$ and $B$ are the endpoints of $\gamma$.
\end{theorem}

\begin{remark}
Since  obviously $\per (K)\geq 2\diam (K)$, the inequality \eqref{eq.main} immediately implies the following widely known inequality:
$\length (\gamma)\geq \frac{1}{2} \per (K)$,
see e.g. \cite{FMP1984}.
\end{remark}

The strategy of our proof is as follows.
We fix a convex figure $K\subset \mathbb{R}^2$. Then we prove the existence of a curve $\gamma_0$
of minimal length among all curves $\gamma$ satisfying the condition
$K \subset \co(\gamma)$ (Section \ref{sect.2}). After that we prove the inequality
$\length (\gamma_0)\geq \per (K) - \diam(K)$ and study all possible cases of the equality
$\length (\gamma_0)= \per (K) - \diam(K)$, where $\gamma_0$ is an arbitrary curve of minimal length among all curves $\gamma$ satisfying the condition
$K \subset \co(\gamma)$ (Section \ref{sect.3}). This allow us to get the proof of Theorem \ref{the1} in Section \ref{sect.4}.

\section{Some auxiliary results}\label{sect.2}

To prove the desired results, we first recall some important properties of curves and convex figures.

Let us recall the following useful definition. A sequence of curves $\{\gamma_i \}_{i\in \mathbb{N}}$
converges uniformly to a curve $\gamma$ if the curves $\gamma_i$ admits parameterizations with the same domain
that uniformly converges to some parametrization of $\gamma$. We will need the following result (see e.g. Theorem 2.5.14 in \cite{BBI2001}):

\begin{prop}[Arzela~--~Ascoli theorem for curves]\label{propAA}
Given a compact metric space. Any sequence of curves which have uniformly bounded lengths has
an uniformly converging subsequence.
\end{prop}

We also note one important property (the lower semi-continuity of length) of the limit curve in the above assertion (see e.g. Proposition 2.3.4 in \cite{BBI2001}).

\begin{prop}\label{propAAn}
Given a sequence of rectifiable curves $\{\gamma_i\}_{i\in \mathbb{N}}$ which converges pointwise to $\gamma$
{\rm(}with respect to parameterizations with the same domain{\rm)}.
Then the inequality $\lim\limits_{i\to \infty} \inf \length (\gamma_i)\geq \length (\gamma)$ holds.
\end{prop}

The following property (of the monotonicity of perimeter) of convex figures is well-known (see e.g. \cite[\S 7]{BoFe1987}).

\begin{prop}\label{monotper}
If convex figures $K_1$ and $K_2$ in the Euclidean plane are such that $K_1\subset K_2$,
then $\per(K_1) \leq \per(K_2)$, and the equality holds if and only if $K_1=K_2$.
\end{prop}

We need also the following well-known result (it could be proved using the Crofton formula, see e.g. \cite[pp. 594--595]{AkVys2017}):

\begin{prop}\label{curvconvhull}
Let $\varphi:[c,d] \rightarrow \mathbb{R}$ be a parametric continuous curve with $\varphi(c)=\varphi(d)$.
Then the length of the curve $\gamma=\{\varphi(t)\,|\, t\in [c,d]\}$ is greater or equal to $\per (\co(\gamma))$.
Moreover, the equality holds if and only if $\gamma$ is closed convex curve.
\end{prop}

Now, we are going to prove the following

\begin{prop}\label{pr.part1}
For a given convex figure $K\subset \mathbb{R}^2$, there is a curve $\gamma_0$
of minimal length among all curves $\gamma$ satisfying the condition
$K \subset \co(\gamma)$.
\end{prop}

\begin{proof} If $\intt(K) =\emptyset$, then the proposition is trivial. In what follows, we assume $\intt(K) \neq \emptyset$.
Denote by $\Delta(K)$ the set of all planar curve $\gamma$ such that $K \subset \co(\gamma)$.
Let us consider $M=\inf \left\{ \length (\gamma)\,|\, \gamma \in \Delta(K) \right\}$.
It is clear that $M\leq \per (K)$, since $\bd(K)$ could be considered as a curve $\gamma$.
Now, we consider the sequence of curves $\{\gamma _i\}_{i\in \mathbb{N}}$ from $\Delta(K)$ such that $\length( \gamma_i) \to M$ as $i \to \infty$.
Without loss of generality we may assume that $\length( \gamma_i) \leq  M+1$ for all $i =1,2,3,\dots$.

Let us take a point $O \in \intt(K)$. There is $r>0$ such that the ball with center $O$ and radius $r$ is inside $K$.
For a fixed $i \in \mathbb{N}$, we consider the point $C_i \in \gamma_i$ such that
$d(C_i,O)=\max \left\{ d(x, O), x \in \gamma_i \right\}$ and the straight line $l_i$  passing through $O$ perpendicular to the straight line $OC_i$.
Since $K \subset \co (\gamma_i)$, there is a point $D_i \in \gamma_i$ such that the line segment $[C_i,D_i]$ intersects $l_i$.
This means that $M+1\geq \length (\gamma_i) \geq d(C_i,D_i)\geq d(C_i,O)\geq r>0$. It implies that $M\geq r >0$ and
$$
\gamma_i \subset B(O,M+1):= \left\{ x\in \mathbb{R}^2, \, d(x,O) \leq M+1 \right\}.
$$
Since the ball  $B(O,M+1)$ is compact
and the lengths of the curves $\gamma_i$, $i =1,2,3,\dots$, are uniformly bounded, then  the sequence $\{\gamma _i\}$ has
an uniformly converging subsequence by Proposition \ref{propAA}.
Passing to a subsequence if necessary, we can assume that the sequence $\{\gamma _i\}_{i\in \mathbb{N}}$ converges uniformly to some curve $\gamma_0$.
Since $K \subset \co(\gamma_i)$ for $i =1,2,3,\dots$, then $K \subset \co(\gamma_0)$ too.
The lower semi-continuity of length (see Proposition \ref{propAAn}) implies
$M=\lim\limits_{i\to \infty} \length (\gamma_i)\geq \length (\gamma_0)$, therefore, $\length (\gamma_0)=M$.
This proves the proposition.
\end{proof}

\begin{center}
\begin{figure}[t]
\centering\scalebox{1}[1]{\includegraphics[angle=0,totalheight=1.5in]
{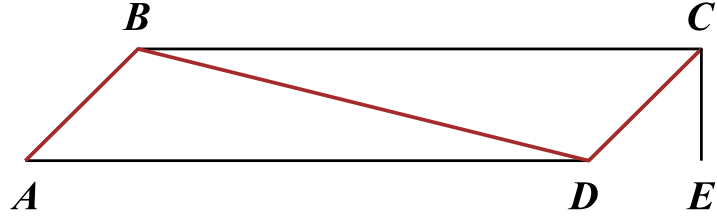}}
\caption{Illustration to Remark \ref{rem.nonc}: a non-convex shortest curve $\gamma$.}
\label{Fig1}
\end{figure}
\end{center}

\begin{remark}\label{rem.nonun}
Note that the curve $\gamma_0$ in Proposition \ref{pr.part1} may not be unique.
For instance, if $K$ is an equilateral triangle, then the union of any two of its sides is such a curve.
\end{remark}

\begin{remark}\label{rem.nonc}
Note also that the curve $\gamma_0$ in Proposition \ref{pr.part1} could be non-convex. For instance, let $K$ be the parallelogram $ABCD\subset \mathbb{R}^2$ with
$A=(0,0)$, $B=(1,1)$, $C=(t+1,1)$, and $D=(t,0)$, where $t\geq 1$. It is easy to see that the broken line $ABCE$ with $E=(t+1,0)$
is one of the shortest convex curves, whose convex hull cover $K$, and its length is $1+\sqrt{2}+t$, see Fig.~\ref{Fig1}.
On the other hand, the length of the broken line $ABDC$ (which convex hull is $K$) is equal to $2\sqrt{2}+\sqrt{2-2t+t^2}$.
It is easy to check that $2\sqrt{2}+\sqrt{2-2t+t^2}< 1+\sqrt{2}+t$ for $t>(3\sqrt{2}+2)/4 \approx 1.56066$.
\end{remark}

The above discussion leads to the following natural problem.

\begin{problem}
Give a comprehensive description of the class of planar curves $\gamma$ with the following property: there is a compact convex figure $K\subset \mathbb{R}^2$
such that $\gamma$ is the shortest curves among all curves, whose convex hulls cover $K$.
\end{problem}

\smallskip

In the next section, we consider a more detail information about any curve
of shortest length among all curves $\gamma$ satisfying the condition
$K \subset \co(\gamma)$ for a given $K$.

\section{Some properties of shortest curves $\gamma$ with $K\subset \co(\gamma)$}\label{sect.3}

Let $U\subset \mathbb{R}^2$ be a convex figure.
We say that a straight line $l\subset \mathbb{R}^2$ divides $U$ into $U_1$ and $U_2$,
if $U_1$ and $U_2$ are convex figures lying in different half-planes relatively $l$, such that $U=U_1\cup U_2$ and $U_1\cap U_2=U\cap l$.

We need the following two simple results.

\begin{lemma}\label{l.div.two}
Let $U\subset \mathbb{R}^2$ be a convex figure and let us consider some points $E,F \in\ext(U)$. Then the straight line $l=EF$
divides $U$ into convex figures $U_1$ and $U_2$ such that $U_i=\co(\ext(U)\cap U_i)$, $i=1,2$.
\end{lemma}

\begin{proof}
It is clear that $\co(\ext(U)\cap U_i)\subset U_i$. Let us suppose that $\co(\ext(U)\cap U_i)\neq U_i$.
Then there is a point $C\in \ext (U_i)$ such that $C \not\in \co(\ext(U)\cap U_i)$. On the other hand,
$\ext (U_i)\subset \ext (U)$ and we obtain the contradiction.
\end{proof}

\begin{center}
\begin{figure}[t]
\centering\scalebox{1}[1]{\includegraphics[angle=0,totalheight=2.0in]
{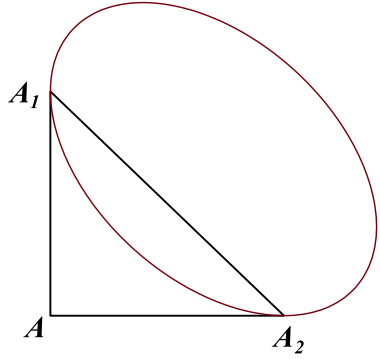}}
\caption{Illustration to Lemma \ref{l.add.point}: the convex figure $U$ and ${\vartriangle}A_1AA_2$.}
\label{Fig1.5}
\end{figure}
\end{center}

\begin{lemma}\label{l.add.point}
Let $U\subset \mathbb{R}^2$ be a convex figure. Let us suppose that a point $A \not\in U$ and points $A_1,A_2 \in U$
are such that the straight lines $AA_1$ and $AA_2$ are support lines for $U$ and $AA_1 \perp AA_2$. Then
$d(A,A_1)+\per(U) > \per \bigl(\co(U\cup\{A\})\bigr)$.
\end{lemma}

\begin{proof}
Let us consider the triangle $A_1AA_2$ and
let $\gamma^{\ast}$ be a part of $\bd(U)$ between the points $A_1$ and $A_2$ such that $U \subset \co\bigl(\gamma^{\ast}\cup\{A\}\bigr)$ (see Fig.~\ref{Fig1.5}).
It is clear that  $\bd \bigl(\co(U\cup\{A\})\bigr)=\gamma^{\ast}\cup[A,A_1]\cup[A,A_2]$.
It is clear also that $\per(U)-\length(\gamma^{\ast})$ is the length of the complementary to $\gamma^{\ast}$ part of $\bd(U)$ between the points $A_1$ and $A_2$, hence
$\per(U)-\length(\gamma^{\ast})\geq d(A_1,A_2)> d(A,A_2)$ and we get
\begin{eqnarray*}
d(A,A_1)+\per(U) > d(A,A_1)+\length(\gamma^{\ast})+ d(A,A_2)= \per \bigl(\co(U\cup\{A\})\bigr),
\end{eqnarray*}
that proves the lemma.
\end{proof}

\smallskip

{\bf Let us fix a curve $\gamma$ with an arc length parametrization $\varphi(t)$, $t\in[a,b]$,
such that $K \subset \co(\gamma)$ and has minimal possible length among all curve with this property.
We put $A:=\varphi(a)$, $B:=\varphi(b)$, and $\widetilde{K}:=\co(\gamma)$.}

\begin{lemma}\label{l.endp}
In the above notations, we have $A, B \in \ext( \widetilde{K})$ and $A\neq B$.
Moreover,  $K\cap [A,B]\neq \emptyset$.
\end{lemma}

\begin{proof}
Let us suppose that $A \not \in \ext( \widetilde{K})$, then there a sufficiently small $\varepsilon>0$ such that
$\varphi([a,a+\varepsilon]) \cap \ext( \widetilde{K}) = \emptyset$ (recall that  $\ext( \widetilde{K})$ is a closed set in $\mathbb{R}^2$).
Hence, if we modify $\gamma$ up to $\gamma_1:=\{\varphi(t)\,|\, t\in[a+\varepsilon,b]\}$ then we get a shorter curve with the same convex hull.
This contradictions shows that $A=\varphi(a) \in \ext( \widetilde{K})$. Similar arguments imply $B=\varphi(\beta) \in \ext( \widetilde{K})$.

Suppose that $B=A$. Let us consider a support line $l_1$ for $\widetilde{K}$ through the point $B$. Since
$B\in \ext(\widetilde{K})$, we may take a point $C \in l_1$ and a support line $l_2$ for $\widetilde{K}$ through $C$, such that $C\not \in \widetilde{K}$
and $l_2$ is perpendicular to $l_1$.
Now, take a point $D\in \widetilde{K}\cap l_2$.
Let $\gamma^{\ast}$ be a part of $\bd(\widetilde{K})$ between the points $B$ and $D$ such that $\widetilde{K} \subset \co\bigl(\gamma^{\ast}\cup[C,D]\bigr)$.
Lemma~\ref{l.add.point} and Proposition~\ref{curvconvhull} imply $d(C,D)+ \length (\gamma^{\ast}) < \per(\widetilde{K})\leq \length(\gamma)$.
Hence, the curve $\gamma^{\ast}\cup[C,D]$ is shorter that $\gamma$, and we get a contradiction due to
$\widetilde{K} \subset \co\bigl(\gamma^{\ast}\cup[C,D]\bigr)$. Therefore, $B\neq A$.

Let us suppose that $K \cap [A,B] =\emptyset$. Then
the distance $\min \{d(x,y)\,|\,x\in K, y \in l\}$ between $K$ and the straight line $AB=:l$
is positive (recall that $K \subset \widetilde{K}$ and $A,B$ are extreme points of $\widetilde{K}$).
Therefore, $K \subset \co \left\{\psi(t)\,|\,t\in [a+\varepsilon, b-\varepsilon] \right\}\subset \co (\gamma)$ for sufficiently small $\varepsilon >0$.
Since the curve $\gamma_2:=\left\{\psi(t)\,|\, t\in [a+\varepsilon, b-\varepsilon] \right\}$  is shorter than $\gamma$, we get a contradiction. This proves that
$K \cap [A,B] \neq \emptyset$.
\end{proof}
\smallskip

\begin{prop}\label{prop.minc.1}
Let us consider $\alpha, \beta \in [a,b]$ such that $\varphi(\alpha), \varphi(\beta) \in\ext( \widetilde{K})$. Then one of the following assertions holds:

1) $[\varphi(\alpha), \varphi(\beta)]\subset \bd(\widetilde{K})$;

2) the straight line $l$ through the points $\varphi(\alpha)$ and $\varphi(\beta)$ divided $\widetilde{K}$ into $\widetilde{K}_1$ and $\widetilde{K}_2$
such that $\left(\widetilde{K}_i\setminus [\varphi(\alpha), \varphi(\beta)] \right)\cap K \neq \emptyset$, $i=1,2$.
\end{prop}

\begin{proof} Let us suppose that $[\varphi(\alpha), \varphi(\beta)]\not\subset \bd(\widetilde{K})$, then every $\widetilde{K}_i$, $i=1,2$,
has a point $C_i$ from $\ext(\widetilde{K})\setminus \{\varphi(\alpha), \varphi(\beta)\}$.
It is clear that $C_i=\varphi(t_0)$ for some $t_0\in [a,b]$.

If $\left(\widetilde{K}_i\setminus [\varphi(\alpha), \varphi(\beta)] \right)\cap K = \emptyset$, then
$K \subset \co \bigl(\ext (\widetilde{K}_j)\bigr)$, $j\in \{1,2\}\setminus \{i\}$, by Lemma~\ref{l.endp}.
Now, we will show how one can modify $\gamma$ to a curve $\gamma_1$ which is shorter than $\gamma$, but $K \subset \co(\gamma_1)$.

If $t_0=a$ ($t_0=b$), then we can take a sufficiently small
$\varepsilon>0$ such that
$\varphi([a,a+\varepsilon]) \cap  l =\emptyset$ (respectively $\varphi([b-\varepsilon,b]) \cap  l =\emptyset$).
Then we see that $K \subset \co \bigl(\ext (\widetilde{K}_j)\bigr)\subset \co(\gamma_1)$, where $\gamma_1=\{\varphi(t)\,|\, t\in[a+\varepsilon,b]\}$
(respectively, $\gamma_1=\{\varphi(t)\,|\, t\in[a,b-\varepsilon]\}$).
Hence, we have found a curve that is shorter than $\gamma$ and which convex hull contains $K$, that is impossible.

If $t_0\in (a,b)$, then we can take $t_1,t_2 \in [a,b]$, $t_1<t_2$,  such that $t_0 \in (t_1,t_2)$ and $\varphi([t_1,t_2])\cap l =\emptyset$.
Since $\varphi(t_0) \in \ext(\widetilde{K})$, then $\varphi([t_1,t_0])\neq [\varphi(t_1), \varphi(t_2)]$.
Now we consider a curve $\gamma_2=\left(\gamma \setminus \varphi([t_1,t_0])\right) \cup [\varphi(t_1), \varphi(t_2)]$.
Obviously, $\gamma_2$ is shorter than $\gamma$, but its convex hull still contains $K$.
This contradiction proves the proposition.
\end{proof}

\begin{corollary}\label{cor.minc.1}
Suppose that  $\varphi(t_0)$ is an extreme point of $\widetilde{K}$ and it is not isolated in the set $\ext(\widetilde{K})$.
Then $\varphi(t_0) \in K$.
\end{corollary}

\begin{proof}
Let us take a sequence $\{t_n\}_{n\in \mathbb{N}}$, $t_n \in [a,b]$, such that all points $\varphi(t_n)$ are extreme for $\widetilde{K}$,
$\varphi(t_n)\neq \varphi(t_0)$, $[\varphi(t_0), \varphi(t_n)] \not \subset \bd(\widetilde{K})$,
and $\varphi(t_n)\to \varphi(t_0)$ as $n\to \infty$.
By Proposition \ref{prop.minc.1}, the straight line $l_n$ through the points $\varphi(t_n)$ and $\varphi(t_0)$ divides $\widetilde{K}$ into two convex figures,
each of them contains some point of $K$. Let $\widetilde{K}_n$ be a one of these two figures, which has a smaller diameter.
It is clear that $\diam(\widetilde{K}_n)\to 0$ as
$n \to \infty$. If $C_n\in \widetilde{K}_n \cap K$, then $C_n \to \varphi(t_0)$ as $n \to \infty$.
Since $K$ is closed, we get $\varphi(t_0)\in K$.
\end{proof}
\smallskip

By Lemma \ref{l.endp}, the points $A$ and $B$ are extreme points of $\widetilde{K}$.
If $A$ (respectively, $B$) is not an isolated point in the set $\ext(\widetilde{K})$, then
$A \in K$ (respectively, $B \in K$). The following proposition deals with the case, when
$A$ (or $B$) is isolated in $\ext(\widetilde{K})$.

\begin{prop}\label{prop.minc.2}
If $A=\varphi(a)$ is isolated in $\ext(\widetilde{K})$, then there are $\tau_1, \tau_2 \in (a,b]$, $\tau_1<\tau_2$, such that
the following assertions holds:

1) $[A,\varphi(\tau_1)]\cup [A,\varphi(\tau_2)]$ covers some neighborhood of $A$ in $\bd(\widetilde{K})$;

2) $\varphi([a,\tau_1])=[A,\varphi(\tau_1)]$,

3) $\varphi([a,\tau_2])\cup [A,\varphi(\tau_2)]$ is a closed convex curve;

4) $[A,\varphi(\tau_1)]\cap K \neq \emptyset$;

5) The angle between the line segments $[A,\varphi(\tau_1)]$ and $[A,\varphi(\tau_2)]$ is equal to $\pi/2$.

\noindent Similar results hold for $B=\varphi(b)$, if it is isolated in $\ext(\widetilde{K})$.
\end{prop}

\begin{proof}
Since the point $A$ is extreme in $\widetilde{K}$ and isolated in $\ext(\widetilde{K})$, then
there are points $A_1, A_2 \in \ext(\widetilde{K})\subset \bd(\widetilde{K})$ such that $[A,A_1], [A,A_2] \subset \bd(\widetilde{K})$ and
$[A,A_1]\cup [A,A_2]$ covers some neighborhood of $A$ in $\bd(\widetilde{K})$
(roughly speaking, $A_1$ and $A_2$ are closest extreme points to $A$ with respect to different directions on $\bd(\widetilde{K})$).
It is clear that $A_1=\varphi(\tau_1)$ and $A_2=\varphi(\tau_2)$ for some $\tau_1, \tau_2 \in (a,b]$. Without loss of generality we may suppose that $0<\tau_1<\tau_2$.

Let us consider the following closed curves:
$$
\gamma_1=\varphi([a,\tau_1])\cup [A,\varphi(\tau_1)], \quad \gamma_2=\varphi([a,\tau_2])\cup [A,\varphi(\tau_2)].
$$

By Proposition \ref{curvconvhull}, we get that $\length(\gamma_1)\geq \per(\co(\gamma_1))$ and $\length(\gamma_2)\geq \per(\co(\gamma_2))$.
Since $[A,A_1], [A,A_2] \subset \bd(\widetilde{K})$, then $[A,A_1]\subset \bd(\co(\gamma_1))$ and $[A,A_2] \subset \bd(\co(\gamma_2))$.
Due to the inclusion $\gamma_i \subset \co(\gamma_i)$, $i=1,2$, we may replace the curve $\gamma$ with
the curve
$$
\gamma_i^-:=\varphi([\tau_i,b]) \cup \bigl(\bd(\co(\gamma_i))\setminus [A,A_i]\bigr)
$$
with the same convex hull $\widetilde{K}$. Since $\gamma$ has minimal length among all curves which convex hull covers $K$, we get
$\length(\gamma_i)=\per (\co(\gamma_i))$ by Proposition \ref{curvconvhull}.
It means that $\gamma_1$ and $\gamma_2$ are closed convex curves by Proposition \ref{curvconvhull} (see Fig.~\ref{Fig2}).

Since $A,A_1 \in \co(\gamma_2)$, then $[A,A_1] \subset \co(\gamma_2)$. On the  other hand, $[A,A_1] \subset \bd (\widetilde{K})$.
Since $\co(\gamma_2)\subset \widetilde{K}$, we get
$[A,A_1]\subset \bd(\co(\gamma_2))$. It implies that $[A,A_1]=\varphi([a,\tau_1])$ and $[A,A_2]\neq \varphi([a,\tau_2])$.
Therefore, assertions 1)--3) are proved.

Let us prove 4). If $[A,\varphi(\tau_1)]\cap K=\emptyset$, then there is $\varepsilon >0$ such that $\co \bigl(\varphi([a,\tau_1+\varepsilon])\bigr)$
and $K$ are situated
in different half-planes with respect to some straight line. Therefore, $K \subset \co(\gamma_3)$, where
$\gamma_3:=\varphi([\tau_1+\varepsilon,b]) \cup [A, \varphi(\tau_1+\varepsilon)]$. On the other hand, $\gamma_3$ is shorter than $\gamma$
(recall that $\varphi(\tau_1)$ is extreme in $\widetilde{K}$, hence $\varphi([a, \tau_1+\varepsilon]) \neq [A, \varphi(\tau_1+\varepsilon)]$).
This contradiction implies $[A,\varphi(\tau_1)]\cap K \neq\emptyset$.

Finally, let us prove 5). If $\angle A_1AA_2 \neq \pi/2$, then we can take $A'\in [A,A_1]$ such that $A'\neq A$ and $d(A,A')$ is less than distance from $A$ to $K$.
Then $A'=\varphi(\tau')$ for some $\tau'\in (a,\tau_1)$. Now, take a point $A'' \in [A,A_2]$ such that $[A', A'']$ is orthogonal to $[A,A_2]$.
If we consider $\gamma_4:=\varphi([\tau',b]) \cup [A',A'']$, then $K \subset \co(\gamma_4)$ and $\length(\gamma_4) <\length(\gamma)$
(since the leg is shorter than the hypotenuse in any right triangle). This contradiction shows that $\angle A_1AA_2 = \pi/2$.

Similar results for the point $B$  we get automatically, reversing the parameterization of the curve $\gamma$.
The proposition is completely proved.
\end{proof}

\begin{center}
\begin{figure}[t]
\centering\scalebox{1}[1]{\includegraphics[angle=0,totalheight=1.7in]
{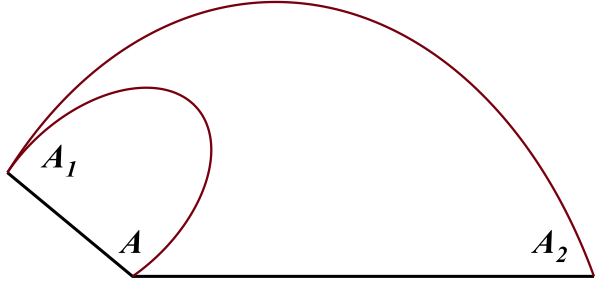}}
\caption{Illustration to the proof of Proposition \ref{prop.minc.2}: the curves $\gamma_1$ and $\gamma_2$.}
\label{Fig2}
\end{figure}
\end{center}

\begin{prop}\label{prop.minc.3}
In the above notations, let $\eta_1$ be the smallest number in $T$
and let $\eta_2$ be the  largest number in $T$, where $T=\{t\in [a,b]\,|\, \varphi(t) \in K\}$.
Then the following inequality holds:
$$
\length(\gamma)+d\bigl(\varphi(\eta_1),\varphi(\eta_2)\bigr) \geq \per(\widetilde{K}) \geq \per(K).
$$
\end{prop}

\begin{proof}
Since $K \subset \widetilde{K}$, then the inequality $\per(\widetilde{K}) \geq \per(K)$ follows directly from Proposition \ref{monotper}.
Therefore, it suffices to prove the inequality
\begin{equation}\label{eq.0}
\length(\gamma)+d\bigl(\varphi(\eta_1),\varphi(\eta_2)\bigr) \geq \per(\widetilde{K}).
\end{equation}

We have $\varphi([a,\eta_1])=[A,\varphi(\eta_1)]\subset \bd(\widetilde{K})$ and
$\varphi([\eta_2, b])=[\varphi(\eta_2), B]\subset \bd(\widetilde{K})$ by Proposition \ref{prop.minc.2}.
Proposition \ref{prop.minc.2} also implies that there is $\theta_1 \in (a,b]$  such that
$[A,\varphi(\eta_1)]\cup [A, \varphi(\theta_1)]$ covers a neighborhood of $A$ in $\bd(\widetilde{K})$ if $A \not \in K$ and
there is $\theta_2 \in [a,b)$ such that
$[B,\varphi(\eta_2)]\cup [B, \varphi(\theta_2)]$ covers a neighborhood of $B$ in $\bd(\widetilde{K})$ if $B \not \in K$
(note that $\theta_1=b$ if and only if $\theta_2 =a$).

Let us consider $\widehat{\gamma}=\varphi([\eta_1,\eta_2])$ and $\widehat{K} =\co(\widehat{\gamma})$.
Note that $\widehat{K}\subset \widetilde{K}$ and $\widehat{K}$ contains all extreme points of $\widetilde{K}$
with the possible exception of points $A$ and $B$ (the latter is possible only if $A$ or $B$ is not in $K$). Therefore,
$\widetilde{K}=\co \bigl(\widehat{K} \cup \{A,B\} \bigr)$.

Since $\widehat{\gamma}\cup [\varphi(\eta_1),\varphi(\eta_2)]$ is a closed curve, Proposition \ref{curvconvhull} implies the inequality
\begin{equation}\label{eq.1}
\length(\widehat{\gamma})+d\bigl(\varphi(\eta_1),\varphi(\eta_2)\bigr)\geq \per(\widehat{K}).
\end{equation}

Let us consider the following {\bf four cases}:
1) $A,B \in K$, 2) exactly one of the points $A$ and $B$ is in $K$, 3) $A,B \not \in K$ and $\theta_1<b$, 4) $A,B \not \in K$ and $\theta_1=b$.

In {\bf Case 1)} we have $\gamma= \widehat{\gamma}$ and $\widehat{K}=\widetilde{K}$, hence \eqref{eq.1} implies
$\length(\gamma)+d\bigl(\varphi(\eta_1),\varphi(\eta_2)\bigr)\geq \per(\widetilde{K})$ and we got what we need.

Let us consider {\bf Case 2)}. Without loss of generality we may assume that $B\in K$ (hence, $\eta_2=b$ and $\varphi(\eta_2)=B$) and $A\not\in K$.
Hence, $\widetilde{K}=\co \bigl(\widehat{K} \cup \{A\} \bigr)$.
Let us consider the triangle $A_1AA_2$, where $A_1=\varphi(\eta_1)$ and $A_2=\varphi(\theta_1)$.
By Proposition \ref{prop.minc.2} we have $\angle A_1AA_2=\pi/2$.
Since  $A_1, A_2, \in \bd (\widetilde{K})\cap \widehat{\gamma}$,
we get that $A_1, A_2 \in \bd (\widehat{K})$.
Then \eqref{eq.1}
and Lemma \ref{l.add.point} imply
\begin{eqnarray*}
\length(\gamma)+d\bigl(\varphi(\eta_1),\varphi(\eta_2)\bigr)= d(A,A_1)+\length(\widehat{\gamma})+d\bigl(\varphi(\eta_1),\varphi(\eta_2)\bigr)\\
\geq d(A,A_1)+\per(\widehat{K})> \per(\widetilde{K}),
\end{eqnarray*}
that proves \eqref{eq.0}.

To deal with {\bf Case 3)},
let us consider the triangles $A_1AA_2$ and $B_1BB_2$, where $A_1=\varphi(\eta_1)$, $A_2=\varphi(\theta_1)$, $B_1=\varphi(\eta_2)$, and $B_2=\varphi(\theta_2)$.
By Proposition \ref{prop.minc.2} we have $\angle A_1AA_2=\angle B_1BB_2=\pi/2$.
Note that $\theta_1 <\eta_2$ and $\eta_1 < \theta_2$. Since  $A_1, A_2, B_1, B_2 \in \bd (\widetilde{K})\cap \widehat{\gamma}$,
we get that $A_1, A_2, B_1, B_2 \in \bd (\widehat{K})$.
Then \eqref{eq.1}
and Lemma \ref{l.add.point} imply
\begin{eqnarray*}
\length(\gamma)+d\bigl(\varphi(\eta_1),\varphi(\eta_2)\bigr)= d(A,A_1)+d(B,B_1)+\length(\widehat{\gamma})+d\bigl(\varphi(\eta_1),\varphi(\eta_2)\bigr)\\
\geq d(A,A_1) +d(B,B_1)+\per(\widehat{K})>
d(A,A_1) +\per\bigl(\co(\widehat{K}\cup \{B\})\bigr)\\
> \per\bigl(\co\bigl(\co(\widehat{K}\cup \{B\})\cup \{A\}\bigl)\bigr) = \per\bigl(\co(\widehat{K}\cup \{A,B\})\bigr)=\per(\widetilde{K}),
\end{eqnarray*}
that proves \eqref{eq.0}.

Finally, let us consider {\bf Case 4)}. In this case we have $[A,B] \subset \bd(\widetilde{K})$, $A_2=B$, and $A=B_2$.
Let as consider the quadrangle $AA_1B_1B$, where $A_1=\varphi(\eta_1)$ and $B_1=\varphi(\eta_2)$.
By Proposition \ref{prop.minc.2} we have $\angle A_1AB=\angle B_1BA=\pi/2$.
Since  $A_1, B_1 \in \bd (\widetilde{K})\cap \widehat{\gamma}$,
we get that $A_1, B_1 \in \bd (\widehat{K})$.

We denote  by $\gamma_3$ a part of $\bd(\widehat{K})$ between $A_1$ and $A_2$ such that $\widetilde{K} \subset \co(\gamma_3\cup \{A,B\})$ (see Fig.~\ref{Fig3}).
It is clear that
$\bd(\widetilde{K})= \gamma_3 \cup[A,A_1]\cup[B,B_1]\cup[A,B]$.
Note that $\per(\widehat{K})-\length(\gamma_3)$ is the length of the curve $\bigl(\bd(\widehat{K})\setminus \gamma_3\bigr)\cup\{A_1,B_1\}$,
connecting the points $A_1$ and $B_1$.
Hence, $\per(\widehat{K})-\length(\gamma_3)\geq d(A_1,B_1)\geq d(A,B)$ and we get
\begin{eqnarray*}
\length(\gamma)+d\bigl(\varphi(\eta_1),\varphi(\eta_2)\bigr)= d(A,A_1)+d(B,B_1)+\length(\widehat{\gamma})+d\bigl(\varphi(\eta_1),\varphi(\eta_2)\bigr)\\
\geq \per(\widehat{K})+d(A,A_1)+d(B,B_1)\\
\geq \length(\gamma_3)+d(A,B)+d(A,A_1)+d(B,B_1)=\per(\widetilde{K}).
\end{eqnarray*}
Hence, we have proved \eqref{eq.0} for all possible cases.
The proposition is proved.
\end{proof}

\begin{center}
\begin{figure}[t]
\centering\scalebox{1}[1]{\includegraphics[angle=0,totalheight=1.9in]
{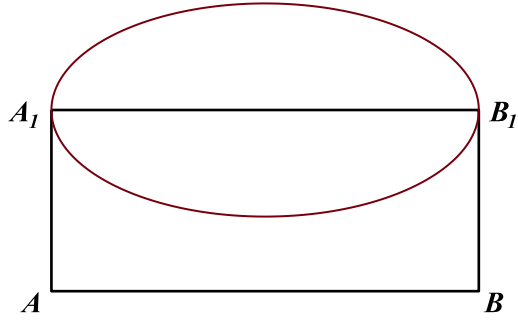}}
\caption{Illustration to Case~4) in the proof of Proposition \ref{prop.minc.3}.}
\label{Fig3}
\end{figure}
\end{center}

\begin{remark}\label{rem.eq.1}
We see from the above proof that  the equality
$$
\length(\gamma)+d\bigl(\varphi(\eta_1),\varphi(\eta_2)\bigr) = \per(\widetilde{K})
$$
is fulfilled if and only if  $\varphi([a,b])\cup [A,B]$ is a convex curve (that coincides with $\bd(\widetilde{K})$)
and the quadrangle $AA_1B_1B$, where $A_1=\varphi(\eta_1)$ and $A_2=\varphi(\theta_1)$, is a rectangle (in particular, $A_1=A$ and $B_1=B$).
Consequently, since $\per(\widetilde{K})=\per(K)$ implies $\widetilde{K}=K$, the equality
$$
\length(\gamma)+d\bigl(\varphi(\eta_1),\varphi(\eta_2)\bigr) = \per(K)
$$
is fulfilled if and only if $\varphi([a,b])\cup [A,B]=\bd(K)$.
\end{remark}

Since $\diam(K)\geq d\bigl(\varphi(\eta_1),\varphi(\eta_2)\bigr)$, then Proposition \ref{prop.minc.3} and Remark \ref{rem.eq.1} imply the following

\begin{corollary}\label{prop.minc.4}
If a curve $\gamma$ has shortest length among all curves whose convex hulls cover a given compact convex figure $K$, then the following inequality holds:
$$
\length(\gamma)+\diam(K) \geq \per(K).
$$
Moreover, the equality in this inequality is fulfilled if and only if $\gamma$ is convex,
$\bd(K)=\gamma \cup [A,B]$,  and $\diam(K)=d(A,B)$, where $A$ and $B$ are the endpoints of the curve $\gamma$.
\end{corollary}

\section{Proof of Theorem \ref{the1}}\label{sect.4}

Let us fix a convex figure $K\subset \mathbb{R}^2$. By Proposition \ref{pr.part1},
there is a curve $\gamma_0$
of minimal length among all curves $\gamma$ satisfying the condition
$K \subset \co(\gamma)$.
By Corollary \ref{prop.minc.4}, we get
$$
\length(\gamma)+\diam(K)\geq \length(\gamma_0)+\diam(K)\geq  \per(K)
$$
for any curve $\gamma$ such that $K\subset \co(\gamma)$,
that proves \eqref{eq.main}. We have the equality in \eqref{eq.main} if and only if $\length(\gamma)=\length(\gamma_0)$ (hence, we may assume that
$\gamma=\gamma_0$ without loss of generality),
$\gamma$ is convex,
$\gamma \cup [A,B]=\bd(K)$,  and $\diam(K)=d(A,B)$, where $A$ and $B$ are the endpoints of the curve $\gamma$.
Therefore, we obtain just convex figures $K$ and corresponding curves $\gamma$ exactly as in Example \ref{ex1}.
\medskip

{\bf Acknowledgements.} The authors would sincerely thank Arseniy Akopyan and Vladislav Vysotsky for reading this paper and helpful discussions.

\vspace{10mm}
\end{document}